\documentclass[a4paper,11pt,times]{amsart}

\usepackage{lineno}
\modulolinenumbers[5]
\usepackage{hyperref}
\usepackage{graphicx}
\usepackage{amsmath,mathrsfs, amssymb,amsthm}

\newtheorem{theorem}{\rm\bf Theorem}[section]
\newtheorem{proposition}[theorem]{\rm\bf Proposition}

\newtheorem{corollary}[theorem]{\rm\bf Corollary}

\theoremstyle{definition}
\newtheorem{definition}[theorem]{\rm\bf Definition}

\theoremstyle{remark}
\newtheorem{remark}[theorem]{\rm\bf Remark}

\newtheorem{example}[theorem]{\rm\bf Example}

\def\scal#1#2{\langle #1, #2\rangle}

\def\R#1{\mathbb{R}^{#1}}

\def\half#1#2{\begin{matrix}\frac{#1}{#2}\end{matrix}}
\def\Iso{\mathscr{I}}

\def\scal#1#2{\langle #1, #2\rangle}

\DeclareMathOperator{\re}{\mathrm{Re}}

\DeclareMathOperator{\trace}{\mathrm{tr}}

\begin{document}

\title[New  explicit solutions to the $p$-Laplace equation]{New  explicit solutions to the $p$-Laplace equation based on isoparametric foliations}

\author{Vladimir G. Tkachev}
\thanks{Supported by G.S. Magnuson's Foundation, grant MG 2017-0101}
\email{vladimir.tkatjev@liu.se}
\address{Link\"oping University, Department of Mathematics, SE-581 83 }


\begin{abstract}
In contrast to an infinite family  of explicit examples of two-dimensional $p$-harmonic functions obtained by G.~Aronsson in the late 80s, there is very little known about the higher-dimensional case. In this paper,  we show how to use isoparametric polynomials to produce diverse examples of $p$-harmonic and biharmonic functions. Remarkably, for some distinguished values of $p$ and the ambient dimension $n$ this yields first examples of rational and algebraic $p$-harmonic functions. Moreover, we show that there are no $p$-harmonic polynomials of the isoparametric type. This  supports a negative answer to a question proposed in 1980 by J.~ Lewis.
\end{abstract}

\keywords{
Isoparametric polynomials; $p$-Laplacian; $\infty$-Laplacian;  biharmonic functions; minimal submanifold; focal varieties; rational solutions}

\maketitle

\section{Introduction}

A foliation of a Riemannian submanifold is called isoparametric if its regular leaves have constant mean curvature. The study and classification of  isoparametric foliations is an important problem of geometry. Remarkably, isoparametric foliations  have been shown to be useful in constructing explicit examples in many problems of analysis, geometry and algebra. We briefly mention their appearance in entire solutions to the minimal surface equation \cite{BGG}, \cite{Simon89}, \cite{SS}, exotic smooth structures \cite{GeJ2016}, Yau conjecture on the first eigenvalue \cite{TangYan13}, nonassociative algebras \cite{Karcher86}, \cite{Tk10a}, \cite{Tk14}, eigenmaps between spheres and Brouwer degrees of gradient maps \cite{Tang2}, \cite{Tang3}, \cite{GeXie}, viscosity solutions to fully nonlinear elliptic PDEs \cite{NTV}, \cite{NTVbook}, Willmore submanifolds \cite{Xie}, integrable structures and mathematical physics \cite{Savo}, \cite{Ferap}.

In this paper, we give yet another application of isoparametric foliations to constructing explicit  examples  of $p$-harmonic and biharmonic functions. We employ an elementary approach  and it is also worth noting that many results of the paper may be easily extended to diverse variational type PDEs. Many of the constructed  examples are either rational or algebraic functions of coordinates.

The existence and construction of homogenous (sometimes called quasiradial or separable) $p$-harmonic functions of the form $|x|^\beta\omega(\theta)$, $\theta=x/|x|$ is well-known and exploited extensively, for example, by L.~V\'{e}ron and S.~Kichenassamy \cite{Veron}, \cite{KSVeron}:   the orthogonal decomposition of the $p$-Laplace operator yields the characteristic equation for $\beta$ and a homogenous first order differential equation for $\omega(\theta)$ on the unit sphere $S^{n-1}\subset\R{n}$. In the two-dimensional case $n=2$, the equation can be solved explicitly in terms of trigonometric functions \cite{Krol73}, \cite{Aronsson86}, \cite{Tk06c}. Some explicit examples of $p$-harmonic functions are also available in \cite{Lind06}, \cite{Lind16a} and for $p=N$ \cite{BorVer07}.

Our approach is somewhat complementary to the aforementioned techniques and relies on the Cartan-M\"unzner equation: we do not work with the orthogonal decomposition of the $p$-Laplace operator, instead we reduce the original equation to a certain second order quasilinear equation, eq. \eqref{plap2} below. One of the benefits of our approach is that it allows to consider a wider (than homogenous functions) class of solutions. On the other hand, due to specific properties of isoparametric polynomials, our construction works only for specific combinations of the parameter $p$ and the ambient  dimension $n$. It would be interesting  to clarify whether these particular combinations are indeed distinguished in an appropriate sense, or they are just a consequence of our method.

 Even in the homogeneous (quasiradial) case, all explicitly known so far higher-dimensional examples correspond to the lowest values of the isoparametric parameter $m=0$ and  $m=1$, see the next section for more details. We extend this on all possible values of the isoparametric parameter $m\in \{1,2,3,4,6\}$.

The paper is organized as follows. In section~\ref{sec54} we discuss the definition and provide some principal examples of isoparametric polynomials. In section ~\ref{sec:plap} we show that the $p$-Laplace equation for a function $f(s,t)$ depending on the distance function $s=|x|^m$ and an isoparametric polynomial $t=\phi(x)$ becomes a quasilinear second order PDE in $s$ and $t$. While classification of solutions to the resulting equation  \eqref{plap2} appears formidably difficult in general, we can report here a complete solution of the problem in some important particular cases, see section~\ref{sec:spec}. We also classify homogeneous solutions of \eqref{plap2}; this can be interpreted as a direct generalization of the quasiradial solutions. In section~\ref{sec:pol}, we prove that there are no  polynomial solutions of isoparametric type.  This  supports a negative answer to a question proposed by J.~ Lewis \cite{Lewis80} on the existence of $p$-harmonic polynomials. Finally, we outline in section~\ref{sec:some} some further applications of our method  to construct  $p$-Laplacian eigenfunctions of the unit sphere and biharmonic functions.

\section{Preliminary facts on isoparametric functions}\label{sec54}
In this section  we recall  some basic concepts and facts on isoparametric functions and isoparametric hypersurfaces, see for example a recent book \cite{CecilRyan2015}. Let $M$ be a Riemannian manifold, $E\subset M$ be an open subset. A smooth function $u : E \to \R{}$ is called \textit{isoparametric} if there exist smooth functions $f$ and $g$ defined on the range of $f$ such that:
$$
|Du|^2=f(u), \quad \Delta u=g(u).
$$
Recall that a (smooth) hypersurface is called {isoparametric} if it is a regular level set of an isoparametric function.

The most interesting case is that of isoparametric hypersurfaces in the real space forms. Then it is well-known that an isoparametric hypersuraface has constant principal curvatures and conversely, any hypersuraface in a space form having constant principal curvatures is a leaf in an isoparametric foliation.
Isoparametric hypersurfaces in Euclidean space $M=\R{n}$ were classified  by Levi-Civita \cite{LevCivita} for $n=3$ and Segre \cite{Segre} for all $n$. At the same time, \`{E}.~Cartan solved the problem in the case of the hyperbolic space. In both cases the number $m$ of distinct principal curvatures is at most 2, and the hypersurfaces are essentially  tubes over a totally geodesic subspace.

The sphere the situation is much more interesting and the problem to classify isoparametric hypersurfaces in the Euclidean spheres $\mathbb{S}^{n-1}\subset \R{n}$ has very deep connections with nonassociative division algebras \cite{Karcher86} and commutative algebra \cite{CCC}, \cite{ChiBook}. Cartan itself classified all  isoparametric hypersurfaces in the Euclidean spheres with $m\le 3$ distinct principal curvatures. According to a celebrated result of F.~M\"unzner \cite{Mun1}, the number $m$ of distinct principal curvatures of an isoparametric hypersurface in a unit sphere can only be $m=1,2,3,4$ or $6$, and any such a hypersurface is the restriction on the unit sphere of a level set $M=\phi^{-1}(t)\cap \mathbb{S}^{n-1}$ ($t\in [-1,1]$) of a \textit{homogeneous polynomial solution} of
\begin{align}\label{Muntzer1}
 |\nabla \phi(x)|^2&=m^2|x|^{2m-2},\qquad x\in \R{n}
 \\ \label{Muntzer2}
 \Delta \phi(x)&=\frac{1}{2}(m_2-m_1)\,m^2|x|^{m-2},\qquad x\in \R{n}.
 \end{align}
where $(m_1,m_2)$ are the multiplicities of distinct principal curvatures related to the ambient dimension $n$ and the multiplicity number $m$ by
  \begin{equation}\label{obstr}
   n=\frac{1}{2}(m_2+m_1)m+2.
  \end{equation}
We emphasize that $m_1$ and $m_2$ are positive integers. As the matter of convention we always assume that $m_1\le m_2$. Notice that $\phi$ is harmonic if and only if $m_1=m_2$.

\begin{definition}\label{def1}
A homogeneous polynomial $\phi$ is said to be \textit{isoparametric} if it satisfies (\ref{Muntzer1})--(\ref{Muntzer2}). We write
$$
\phi\in \mathscr{I}_m(m_1,m_2)
$$ if $\phi$ satisfies both (\ref{Muntzer1}) and (\ref{Muntzer2}), and $\phi\in \mathscr{I}_m,$ if rather $\phi$ satisfies  (\ref{Muntzer1}).
\end{definition}

Notice that a homogeneous polynomial solution to the first Cartan-M\"unzner equation (\ref{Muntzer1}) alone is a composition of some isoparametric form with a Chebyshov polynomial \cite{Tk14b}.

Below we consider some explicit representations of isoparametric polynomials. Let us make some remarks concerning notations. By $\phi_{m,m_1,m_2}$ we denote an isoparametric polynomial of degree $m$ and having the multiplicities $(m_1,m_2)$. Note that this notation may be ambiguous in certain cases when $m=4$ and $m_1\ne m_2$.  In the examples below, we give corresponding explicit representations in some specific Euclidean coordinates providing a most optimal form. The reader have to note, however, that an isoparametric polynomial is determined up to an orthogonal transformation of the ambient space.

The case $m=1$ is trivial: any linear function of $x\in \R{n}$ is an isoparametric polynomial.
It is also straightforward to verify that for $m=2$ the only possible  isoparametric quadratic forms are
\begin{equation}\label{m2}
\phi=x_1^2+\ldots+x_{m_2}^2+x_{m_2+1}^2 -x_{m_2+2}^2-\ldots-x_{m_1+m_2+2}^2, \quad m_1, m_2\in \mathbb{Z}^{+}.
\end{equation}
If $m_1=m_2$ then the above function  is harmonic and the zero level set $\phi(x)=0$ is a Clifford-Simon minimal cone.

By the classical Cartan result \cite{Cartan38},  there are exactly four (up to an isometry) isoparametric polynomials for $m=3$. Furthermore, they are \emph{harmonic} and the multiplicities of the principal curvatures coincides with the dimensions of the classical division algebras:   $m_1=m_2\in \{1,2,4,8\}$. More precisely,  the corresponding  cubic forms in dimensions $n=5,8,14$ and $26$  are given by
\begin{equation}\label{CartanFormula}
\begin{split}
\phi(x)=&x_{n}^3+\frac{3}{2}x_{n}(|z_1|^2+|z_2|^2-2|z_3|^2-2x_{n-1}^2)+ \frac{3\sqrt{3}}{2}x_{n-1}(|z_1|^2\\
&-|z_2|^2)+{3\sqrt{3}}\re z_1(z_2z_3)\qquad\qquad d=1,2,4,8,
\end{split}
\end{equation}
where $z_k=(x_{kd-d+1},\ldots,x_{kd})\in \R{d}\cong\mathbb{A}_d$, $k=1,2,3$, and  $\mathbb{A}_d$ denotes the real division algebra of dimension $d$: $\mathbb{A}_1=\mathbb{R}$ (the reals), $\mathbb{A}_2=\mathbb{C}$ (the complexes), $\mathbb{A}_4=\mathbb{H}$ (the quaternions) and $\mathbb{A}_8=\mathbb{O}$ (the octonions). In fact, one can rewrite (\ref{CartanFormula}) in a more compact way as follows
\begin{equation}\label{Iso3det}
\phi(x)=-\frac{\sqrt{3}}{2}\trace T^3,
\quad\quad
T=\left(
            \begin{array}{ccc}
              \half{1}{\sqrt{3}}x_{n}+x_{n-1} & z_3 & \bar z_2 \\
              \bar z_3& \half{1}{\sqrt{3}}x_{n}-x_{n-1} & x_1 \\
              z_2 & \bar z_1 & \half{-2}{\sqrt{3}}x_{n} \\
            \end{array}
          \right),
\end{equation}
where the trace representation for $d=4,8$ should be understood in the  Jordan algebra sense \cite{BS2}, \cite{Tk14}.

When  $m=4$, the situation is much more involved. The results of R.~Takagi \cite{Takagi}, Ozeki and Takeuchi \cite{OT1}, \cite{OT2}, S.~Stolz \cite{Stoltz}, Cecil, Chi and Jensen \cite{CCC} and Q.-S.~ Chi \cite{ChiBook} establish an ultimate classification of isoparametric hypersurfaces with four principal curvatures. In summary, any isoparametric hypersurface with $m=4$ is either from the Ferus-Karcher-M\"unzner family \cite{FKM} based on the representations of Clifford algebras (see \eqref{FKM} below), or homogeneous (i.e. the orbits of certain subgroups
of the orthogonal group $O(n)$) with $(m_1,m_2)=(2,2)$ or $(4,5)$ in $\R{10}$ and $\R{20}$, respectively. \'{E}.~Cartan \cite{Cartan40} was the first to classify all such isoparametric hypersurfaces. He established that $m_1=m_2\in \{1,2\}$. The corresponding explicit representations in $\R{6}$ and $\R{10}$ are given respectively by
\begin{equation}\label{m411}
\phi=|x|^4+|y|^4+8\scal{x}{y}^2-6|x|^2|y|^2,
\end{equation}
where $x=(x_1,x_2,x_3)$, $y=(x_4,x_5,x_6)$, and
\begin{equation}\label{m422}
\phi=\frac{1}{2}(\trace X^4-\frac{3}{8}(\trace X^2)^2),
\quad X=\left(
    \begin{array}{rrrrr}
      0 & x_1 & x_2 & x_3 & x_4 \\
      -x_1 & 0 & x_5 & x_6 & x_7 \\
      -x_2 & -x_5 & 0 & x_8 & x_9 \\
      -x_3 & -x_6 & -x_8 & 0 & x_{10} \\
      -x_4 & -x_7 & -x_9 & -x_{10} & 0 \\
    \end{array}
  \right),
\end{equation}
see \cite{OT1}, \cite{OT2}. The first example is a member of the so-called FKM-isoparamteric (or the Ferus-Karcher-M\"unzner) polynomials parameterized by symmetric Clifford systems as follows. A finite set $\mathcal{A}=\{A_i\}_{1\le i\le q}$ of real symmetric $2p\times 2p$-matrices  $A_i$ is called a {\textit{symmetric Clifford system}} if
\begin{equation}\label{Apolar1}
A_iA_j+A_jA_i=2\delta_{ij}I,
\end{equation}
where $I$ is the unit matrix. The simplest example is for $q=2$, $p=1$:
$$
A_1=\left(
        \begin{array}{cc}
         1 & 0 \\
          0 & -1\\
        \end{array}
      \right),
      \quad
A_2=\left(
        \begin{array}{cc}
         0 & 1 \\
          1 & 0\\
        \end{array}
      \right).
$$
A symmetric Clifford system of size $q$ in $\R{2p}$ exists if and only if the inequality
\begin{equation}\label{qp}
q\le 1+\rho(p)
\end{equation}
holds,
where  the Hurwitz-Radon function $\rho$ is defined by
\begin{equation*}\label{foll}
\rho(s)=8a+2^b, \qquad \text{if} \;\,\,s=2^{4a+b}\cdot \mathrm{odd} , \;\; 0\leq b\le 3,
\end{equation*}
see \cite{Shapiro}, \cite{CecilRyan2015}. Then the quartic form
\begin{equation}\label{FKM}
\phi_{\mathcal{A}}=|x|^4-2\sum_{i=1}^q(x^tA_ix)^2
\end{equation}
is an isoparametric polynomial in $\R{2p}$,  see \cite{CecilRyan2015} for more details. Then the isoparametric parameters (multiplicities of the principal curvatures) are given by
\begin{equation}\label{m1m2m}
m_1=q-1, \qquad m_2=p-q.
\end{equation}
Note also a nice relation between FKM-isoparamteric polynomials and cubic minimal cones arising from Clifford algebras, see \cite{Tk10c}:  the zero level set of the cubic form
\begin{equation}\label{FKM4}
f(z)=\sum_{i=1}^q (x^tA_ix)y_i, \qquad z=(x,y)\in \R{2p}\times \R{q},
\end{equation}
is a minimal cone in $\R{2p+q}$.

Finally, if $m=6$ then by the results of Abresh \cite{Abresh} we know that $m_1=m_2\in\{1,2\}$, and then Dorfmeister and Neher \cite{DorfN}  and Miyaoka \cite{Miyaoka13} have proved that there exist only two isoparametric forms of degree $6$ in dimensions $8$ and $14$, respectively and the corresponding isoparametric hypersurfaces are homogeneous (see also \cite{Siffert3} which simplifies the Dorfmeister and Neher's result and refines the Miyaoka's argument). Explicit matrix representations can be found  in \cite{PengHou}.

\begin{table}[h]
  \centering
  \begin{tabular}{c|l|l}
  $m$\quad  &\quad  $m_1$ &\quad  $n=mm_1+2$\\\hline
  $2$ \quad &\quad   $1,2,3,\ldots$ &\quad $2m_1+2$ \\\hline
  $3$ \quad &\quad   $1,2,4,8$ &\quad $3m_1+2$\\\hline
  $4$ \quad &\quad   $1,2$ &\quad $4m_1+2$\\\hline
  $6$ \quad &\quad   $1,2$ &\quad  $6m_1+2$\\
\end{tabular}
\caption{The possible multiplicities $m_1=m_2$ and the ambient dimension $n$}\label{tab1}
\end{table}

\section{The isoparametric  Ansats}\label{sec:plap}

It is the well-known and frequently exploited fact that radial symmetric solutions encode the most fundamental properties of many variational type PDEs, in particular those invariant under the action of the orthogonal group $O(n)$.  It is the also a particular aim of the present paper  to extend this observation onto a  subclass of solutions based on isoparametric polynomials. As the main model example, we consider the $p$-Laplace operator including the limit case $p=\infty$.
The method used below, however, applies also to general $O(n)$-equivariant operators, such as the mean curvature equation or a general Hessian equation.

Given $p\in \R{}$, the operator
\begin{equation}\label{plaplace}
\Delta_p u:=|\nabla  u|^2\Delta u+\half{p-2}{2}\nabla u\cdot \nabla |\nabla u|^2=0
\end{equation}
is called the $p$-Laplacian. Here $u(x)$ is a function defined on a domain $E\subset \R{n}$, $\nabla u$ is its gradient and $\cdot$ denotes the standard inner product in $\R{n}$. A classical $C^2$-solution of \eqref{plaplace} is called a $p$-harmonic function.
In general, when  $p>1$ and $p\ne2$, $u$ should be understood as a weak (in the distributional sense) solution to (\ref{plaplace}). Then $u$ is normally in the class $C^{1,\alpha}(E)$ \cite{Ural68}, \cite{Uhlenbeck}, \cite{Evans82}, but need not to be a H\"older continuous or even continuous in a closed domain with nonregular boundary \cite{KrolMaz72}. On the other hand, if $u(x)$ is a weak solution of (\ref{plaplace}) such that  $\mathrm{ess} \sup |\nabla u(x)|>0$ holds locally in a domain $E\subset \R{n}$ then $u(x)$ is in fact a real analytic function in $E$ \cite{Lewis77}.

The natural question to ask is then:  If there exist (nontrivial) \textit{rational} or \textit{algebraic} examples of $p$-harmonic functions for $p\ne2$ (including the case $p=\infty$)? In $\R{2}$ the answer is `yes' as it follows from the quasiradial examples  constructed by G.~Aronsson in \cite{Aronsson86}, \cite{Aronsson89} and its further analysis given by the author in \cite{Tk06c}. In this section (see also Remark~\ref{rem:rat}) we present several rational and algebraic examples of $p$-harmonic function.

A very related but much more difficult question is whether there exist nontrivial \textit{homogeneous polynomial} $p$-harmonic functions for $p\ne2$? This problem naturally emerges in connection to the nonvanishing propery of analytic $p$-harmonic functions and  was proposed and studied by J.~Lewis in ~\cite{Lewis80}. In particular, Lewis wa able to establish the negative answer for $n=2$. Very recently some higher-dimensional generalizations for degree $=3,4,5$ were obtained in \cite{Tk16pLapl}, \cite{Lewis2016}. We refer to section~\ref{sec:pol} below for some particular results in this direction.

The radial symmetric $p$-harmonic functions are well-known:
\begin{equation}\label{fundam}
E_{p,n}(x)=
\left\{
\begin{array}{cc}
  (n-p)|x|^{\frac{p-n}{p-1}} & \text{for $p\ne n$} \\
  -\log |x| & \text{for $p= n$}
\end{array}
\right.
\end{equation}
and play the role of the fundamental solution to (\ref{plaplace}). Notice that $E_{p,n}(x)$ is of class $W^{1,p}(\Omega)$ for any domain omitting origin.

In this paper, we study solutions $u(x)$ of \eqref{plaplace} having a general form
\begin{equation}\label{uharm}
u(x)=f(|x|^m,\,\phi(x))
\end{equation}
where $f=f(s,t)$ is a $C^2$-function and $\phi\in \Iso_m(m_1,m_2)$.

\begin{remark}
It seems natural to relax the M\"unzner-Cartan equations (\ref{Muntzer1})--(\ref{Muntzer2}) on $\phi$ and consider instead an  \emph{a priori} bigger set of $p$-harmonic functions of the form (\ref{uharm}) with $\phi$ satisfying the first M\"unzner-Cartan equation (\ref{Muntzer1}) only. Note, however,  that this does not yield some further examples. Indeed, it follows from \cite{Tk14b} that there exists a Chebyshov polynomial $T_k(t)$ and $\psi\in \Iso_{m'}(m'_1,m'_2)$ such that $\phi(x)=|x|^{m'} T_k\bigl(\psi(x)|x|^{-m'}\bigr)$, which implies that  $f(\phi,|x|^m)=F(\psi,|x|^{m'})$ for some suitable $F$ depending on $f$.
\end{remark}

\begin{remark}\label{rem:angle}
Functions given by \eqref{uharm} have a clear geometric interpretation: it is known \cite[Sec.~3.5]{CecilRyan2015} that if $\phi(x)\in \Iso_m$ then there exists a function $\theta(x)$ smooth on $S^{n-1}\setminus M(\phi)$ and $\theta(x):S^{n-1}\to [-1,1]$, such that
\begin{equation}\label{anglem}
\phi(x)=|x|^m\cos m\theta(x)
\end{equation}
In this setup, $\theta(x)$ is naturally understood as a certain (isoparametric) angle function on the unit sphere. Thus, \eqref{uharm} expresses the fact that $u$ depends only the radius $r$ and the `isoparametric angle' $\theta$. In the trivial case $n=2$, we have the standard polar angle $\theta=\arctan\frac{y}{x}$ which amounts to the quasiradial solutions studied by Aronsson \cite{Aronsson86}.
\end{remark}

\begin{proposition}\label{prop1}
Let $\phi\in \Iso_m(m_1,m_2)$ and
\begin{equation}\label{st}
s=|x|^m, \quad t=\phi(x).
\end{equation}
 Then for any $C^2$-function $f=f(s,t)$
\begin{equation}\label{plap1}
\Delta_p f(|x|^m,\,\phi(x)) =m^{4}s^{1-4/m}(sAC+\half{p-2}{2}sB+\half{(m-2)(p-2)}{2m}Ah),
\end{equation}
where
\begin{equation}\label{notation}
\begin{split}
   h&=f_ss+f_tt,  \\
   A&= (f_s^2+f_t^2)s+2f_sf_tt,\\
   B&= (f_sA_s+f_tA_t)s+(f_sA_t+f_tA_s)t,\\
   C&= (f_{ss}+f_{tt})s+2f_{st}t+(\mu+1)f_s+\nu f_t,\\
   \mu&=\half1{2}(m_1+m_2),\\
\nu&=\half1{2}(m_2-m_1),\\
\end{split}
\end{equation}
In particular, $u(x)=f(|x|^m,\,\phi(x))$ is $p$-harmonic if and only if
\begin{equation}\label{plap2}
sAC+\frac{p-2}{2}sB+\frac{(m-2)(p-2)}{2m}Ah=0.
\end{equation}
\end{proposition}

\begin{proof} Using the Euler homogeneity function theorem and  (\ref{Muntzer1}) we have  $\scal{\nabla s}{\nabla t}=m^2s^{1-2/m}t$ and $|\nabla s|^2=|\nabla t|^2=m^2s^{2-2/m}$. This yields  for any two functions $f(t,s)$ and $g(t,s)$ that
\begin{equation}\label{scalgrad}
\scal{\nabla f}{\nabla g}=m^2s^{1-2/m}((f_s^2+f_t^2)s+2f_sf_tt),
\end{equation}
hence
\begin{equation}\label{Df}
|\nabla f|^2=m^2s^{1-2/m}((f_s^2+f_t^2)s+2f_sf_tt)=m^2s^{1-2/m}A.
\end{equation}
Applying \eqref{scalgrad}, one readily obtain in the notation (\ref{notation})  that
\begin{align}
\nabla f\cdot \nabla |\nabla f|^2
&=m^4s^{1-4/m}(sB+(1-\frac{2}{m})Ah)\nonumber\\
\end{align}
and similarly that
\begin{align}
\Delta f&
=m^2s^{1-2/m}\left((f_{ss}+f_{tt})s+2f_{st}t+\frac{m+n-2}{m}f_s +\frac{m_2-m_1}{2}f_t\right)\nonumber\\
&=m^2s^{1-2/m}\bigl((f_{ss}+f_{tt})s+2f_{st}t+(\mu+1)f_s +\nu f_t\bigr)\label{laplce}\\
&=m^2s^{1-2/m}C.\nonumber
\end{align}
Combining the obtained relations with (\ref{plaplace}) finishes the proof.
\end{proof}


A complete analysis of the nonlinear partial differential equation (\ref{plap2}) is a rather difficult problem. In this paper, we make the first step and consider the special solutions satisfying the following ansats:
\begin{equation}\label{ansats}
f(s,t)=s^kg(z), \quad z=\frac{t}{s}.
\end{equation}
In the terminology of \cite{Aronsson89}, \cite{Porr}, \eqref{ansats} describes  the so-called \textit{separable} $p$-harmonic functions, i.e. those having the form
\begin{equation}\label{ansats1}
u(x)=|x|^k\cdot g\bigl(\phi(x)|x|^{-m})\equiv |x|^k g(\cos m\theta),
\end{equation}
where the isoparametric angle is defined as in Remark~\ref{rem:angle}. By virtue of \eqref{anglem},
\begin{equation}\label{z1}
|z|\le 1.
\end{equation}
Then Proposition~\ref{prop1} readily yields

\begin{corollary}
\label{cor:homogen:m1}
The function \eqref{ansats1} is $p$-harmonic if and only if
\begin{equation}\label{geq}
\begin{split}
(z^2-1)\biggl(b_1(z^2-1)g'^2g''+b_2g^2g''+ (b_3z+\nu)g'^3+b_4g'^2g\biggr)+(b_5z-\nu)g'g^2 +b_6 g^3=0,
\end{split}
\end{equation}
where
\begin{equation}\label{alpha}
\begin{split}
&b_1=1-p,\\
&b_2=k^2,\\
&b_3=1-p-\mu,\\
&b_4=k(2kp-3k+\half{n-p}{m}), \\
&b_5=k^2(\mu+1), \\
&b_6=-k^3(kp-k+\half{n-p}{m}).
\end{split}
\end{equation}
\end{corollary}

%
%


\section{Some integrable cases of \eqref{geq}}\label{sec:spec}
Even the full analysis of the nonlinear equation \eqref{geq} is out the scope of the  paper. We  confine ourselves to  some particular  cases of \eqref{geq}, more precisely:

\begin{itemize}
\item[(i)]
for certain particular values of `isoparametric parameters' $m_1,m_2$ and $m$;
\item[(ii)]
for some fixed values of $k$;
\item[(iii)]
some specific ansatz for the function $g$.
\end{itemize}

Some preliminary remarks are in order. Note that in the simplest case  $m=1$ the isoparametric polynomial $\phi(x)$ is essentially one-dimensional and can be written as $\phi(x)=x_1$ in an appropriate orthogonal coordinate system. To be consistent with \eqref{obstr} and \eqref{Muntzer1}--\eqref{Muntzer2}, one obtains that  $m_1=m_2=n-2$, hence in the notation of Corollary~\ref{cor:homogen:m1} this yields $\mu=n-2$ and $\nu=0$. Solutions under this form were first considered and studied by I.~Kroll' in \cite{Krol73}; some further modifications can be made in the dimension $n=3$, see \cite{KrolMaz72}, \cite[p.~359]{Veron88}. On the other hand, the corresponding differential equations  do not yield any explicit solution for $n\ge 3$. We refer to \cite{Krol73} for more details.

\subsection{The case of arbitrary $k$ and $g=bz+a$}\label{sublinear}
This is the first case that yields new nontrivial explicit solutions and works well for all possible isoparametric parameters. More precisely, we consider solutions given by \eqref{ansats} with $g$ satisfying the linear Ansats:
$$
g(z)=z+a.
$$
Note that by the homogeneity, we may without loss generality consider a monic polynomial ($b=1$). Then a simple analysis shows that if $p\ne 2$ then nontrivial solutions emerge only if either $g(z)=z\pm 1$ or $g(z)=z$. In the former case one has $m=2$, $k=1$ and $p=1-m_1$ which is unsatisfactory from the analytic point of view  because $p<0$. But the latter case $g(z)=z$ contains some nontrivial solutions with $p>1$ as the proposition below shows.


\begin{proposition}\label{pro:z}
Let  $\phi\in \mathscr{I}_m(m_1,m_2)$. Then $u(x)=\phi(x)|x|^{(k-1)m}$ is $p$-harmonic if and only if
\begin{equation}\label{m1m2}
m_1=m_2
\end{equation}
and $k\in \{-1,0,1\}$. The corresponding solutions are given by
$$\begin{array}{lclcl}
\bullet\,\, k=1 &\quad& u=\phi(x) &\quad& p=2;\\
\bullet\,\, k=0 &\quad& u=\phi(x)|x|^{-m} &\quad& p=1-m_1\\
\bullet\,\, k=-1 &\quad& u=\phi(x)|x|^{-2m} &\quad& p=2+\frac{2m_1m}{m+1}\\
\end{array}
$$
When $k=-1$, the distinguished dimensions $n$ and the corresponding rational $p$-harmonic functions are given by
\begin{table}[h]
  \centering
  \renewcommand\arraystretch{1.3}
\begin{tabular}{l||l|l|lcl||l|l|l|l|l|l|l|l}
$m$ & $2$ & $2$ & \ldots & $2$& \ldots & $3$& $3$ & $3$& $3$&$4$& $4$& $6$ &$6$    \\\hline
$n$ & $4$ & $6$ & \ldots & $2k$& \ldots & $5$& $8$& $14$& $26$ & $6$& $10$ &$8$ &$14$   \\\hline
$p$ & $\frac{10}3$ & $\frac{14}3$ &\ldots & $\frac{2(2k+1)}3$& \ldots &
$\frac{7}2$ &$5$ & $8$& $14$& $\frac{18}5$& $\frac{26}5$ & $\frac{26}7$ &$\frac{38}7$ \\
\end{tabular}
\end{table}
\end{proposition}

\begin{proof}
Setting $g(z)=z$ in \eqref{geq} yields that $\nu=0$ (implying \eqref{m1m2}) and also that $\beta_3+\beta_4=\beta_5+\beta_6=0$. This implies by virtue of $m\ne0$ that the only solutions are given by $k(1-k^2)=0$ which readily yields the desired conclusion.
\end{proof}

\begin{remark}\label{rem:rat}
From the analytic point of view, the case $k=-1$ is the most interesting because we have $p>2$. In fact, since
$$
n-p=\frac{m-1}{m+1}m_1m>0
$$ the stronger inequality $2<p<n$ holds. The most spectacular observation here is that \textit{$u(x)=\frac{\phi(x)}{|x|^{2m}}$ is a rational $p$-harmonic function.}
\end{remark}

\begin{example}
We illustrate the case $k=-1$ in Proposition~\ref{pro:z} with some examples. As we already pointed out, all these examples are rational functions.  First consider Clifford cone type isoparametric solutions for $m=2$. Then for an arbitrary $k=1,2,3,\ldots$, the rational function
$$
u(x)=\frac{x_1^2+\ldots+x_k^2-x_{k+1}^2-\ldots-x_{2k}^2} {(x_1^2+\ldots+x_{2k}^2)^2}
$$
is a homogeneous order $-2$ rational solution of the $p$-Laplace in $\R{2k}$ with $p=\frac{4k}{3}+2$. Next, for $m=3$ we have that
$$
u=\frac{x_{5}^3+\frac{3}{2}x_{5}(x_1^2+x_2^2-2x_3^2-2x_{4}^2)+ \frac{3\sqrt{3}}{2}x_{4}(x_1^2 -x_2^2)+{3\sqrt{3}}x_1x_2x_3}{(x_1^2+\ldots+x_{5}^2)^3}
$$
is a homogeneous order $-3$ rational solution of the $p$-Laplace in $\R{5}$ with $p=\frac{7}{2}$. Finally, let us mention an explicit example of homogeneous degree $-4$ quartic in $\R{6}$ with $p=\frac{18}{5}$ corresponding to \eqref{m411}:
$$
u=\frac{|x|^4+|y|^4+8\scal{x}{y}^2-6|x|^2|y|^2}{(|x|^2+|y|^2)^4},
$$
where $x=(x_1,x_2,x_3)$, $y=(x_4,x_5,x_6)$.
\end{example}


\subsection{The case $k=0$, $m_2=m_1$ and arbitrary function $g$}

In this case, we have the following complete description of the function $g(z)$.

\begin{proposition}
Let  $\phi\in \mathscr{I}_m(m_1,m_1)$. Then $g(\phi(x)/|x|^{m})$ is $p$-harmonic iff
\begin{equation}\label{alpha1}
g'(z)=C(1-z^2)^{-\alpha}, \qquad p=1+\frac{m_1}{2\alpha-1}.
\end{equation}
\end{proposition}

\begin{proof}
By the assumption, we have $k=0$ in \eqref{ansats}. Then $\beta_2=\beta_4=\beta_5=\beta_6=0$, thus \eqref{geq} amounts to
\begin{equation}\label{geqk0}
\bigl((1-p)(z^2-1)g''+(1-m_1-p)zg'\bigr)g'^2=0.
\end{equation}
Eliminating the trivial case $g=\mathrm{const}$, one finds that
$
g''/g'=\frac{m_1+p-1}{p-1}\cdot \frac{z}{1-z^2}
$
which implies \eqref{alpha1}.
\end{proof}

Specializing $\alpha$, one obtains some interesting particular cases. Furthermore,  if  $\phi\in \mathscr{I}_m(m_1,m_1)$ then one has the following $p$-harmonic functions:

$$\begin{array}{lclcl}
(\mathrm{i})\,\, u=\mathrm{arctanh}\,(\phi(x)|x|^{-m}) &\quad& p=1+m_1 &\quad& n=mm_1+2\\
\\
(\mathrm{ii})\,\, u=\frac{\phi(x)}{\sqrt{|x|^{2m}-\phi(x)^2}} &\quad& p=1+\frac12m_1 && n=mm_1+2\\
\\
(\mathrm{iii})\,\, u=\arcsin \frac{\phi(x)}{|x|^{m}},&\quad& p=\infty && n=mm_1+2,\\
\\
(\mathrm{iv})\,\, u=\ln \frac{|x|}{|y|},\,\, (x,y)\in \R{k}\times \R{k} &\quad& p=k &&n=2k, \,\,k=1,2,3\ldots
\end{array}
$$

\begin{proof}
Specializing $\alpha=1$, $\alpha=\frac32$ and  $\alpha=\frac12$ yields (i), (ii)  and (iii) respectively. The last case is a corollary of (i) corresponding to the quadratic isoparametric polynomials \eqref{m2}. In that case, we have $k-1=m_1=m_2$ and using (i), we see that  the function $\mathrm{arctanh}\frac{|x|^2-|y|^2}{|x|^2+|y|^2}=\ln \frac{|x|}{|y|}$ is $k$-harmonic in $\R{2k}$.
\end{proof}

\subsection{Two-zone anisotropic isolated singularities}
The quasiradial $p$-harmonic functions play a fundamental role in the description of general isolated interior or boundary singularities. In particular, the following famous  result \cite{Serrin65} by Serrin \footnote{This result holds for any quasilinear divergence form equation}   says that a positive singular solution of the  $p$-Laplacian equation should behave as the fundamental solution $E_{p,n}(x)$: If $u$ is a continuous positive solution of $\Delta_p u=0$ in the punctured ball $B\setminus\{0\}$ then either $0$ is a removable singularity (i.e. $u$  can be extended to the whole $B$) or
$$
u \sim |x|^{\frac{p-n}{p-1}} \quad \text{for }x\rightarrow 0.
$$

It is interesting to ask if there is an analogues result for  solutions $u$ of the $p$-Laplace equation with anisotropic isolated singularities. In particular, it is natural to consider the following situation: Let $u$ be a solution of the $p$-Laplace equation with isolated singularity at the origin and satisfying the following \textit{two-zone anisotropic property}: for any $\epsilon>0$ small enough, the sets $$
U^\pm_\epsilon:=\{x\in \R{n}: \pm u(x)>0 \text{ and }0<|x|<\epsilon\}
$$
are connected. What can be said about the asymptotic behavior of $u$? When the origin is a removable singularity? It is also interesting to  characterize the geometry of the zero locus $U^0:=\{x:u(x)=0\}$ nearby the origin in this case.

In two dimensions, the situation is well-studied and  according to \cite{Manfredi}, any solution of the $p$-Laplace equation in $\R{2}$ with an isolated singularity at the origin  behaves as the corresponding quasiradial solution. Also precise asymptotic representation near the singularity has been obtained in \cite{Manfredi}.

In higher dimensions, there  exist of plenty of anisotropic singularities, see for example \cite{Veron}, however, little is known so far about analogues of the Serrin result for solutions with anisotropic singularities. Below we suggest some heuristic argument  supporting the following conjecture.

\medskip
\noindent
\textbf{Conjecture.}
If  $u$ is  a solution of the $p$-Laplace equation with a two-zone isolated singularity at the origin such that
$$
\limsup_{x\to 0}|\phi(x)|\cdot |x|^{\frac{2n-p-2}{p-2}}=0
$$
then $0$ is a removable singularity.

\medskip

Let us explain the appearance of the exponent $\alpha=\frac{2n-p-2}{p-2}$.  First note that the constructed in Proposition~\ref{pro:z}  solutions with $k=-1$ satisfy the two-zone property. Indeed, it is well-known fact  (see, for example, \cite[Sec.~3.5]{CecilRyan2015}) that given an isoparametric polynomial $u(x)$, its zero locus $$
M:=\{x\in\R{n}:|x|=1 \text{ and } u(x)=0\}
$$
is a smooth embedded  constant mean curvature hypersurface of the unit sphere $S^{n-1}\subset\R{n}$. The complete zero locus $U^0=\{x: u(x)=0\}$ is the cone over $M$. Then it is known that $M$ divides the sphere $S^{n-1}$ into two connected components (respectively, $U_0$ divides $\R{n}$ into two components nearby the origin). In fact, this property readily follows  from the eiconal equation \eqref{Muntzer1}. We emphasize that this  two-zone property property holds true for all admissible isoparametric parameters $m,m_1,m_2$ and the dimension $n$.

Now, let us consider the corresponding solution $u_\phi(x)=\phi(x)|x|^{-2m}$ with $\phi\in \mathscr{I}_m(m_1,m_1)$. Then $u$ is singular at the origin and  satisfies the asymptotic growth condition
$$
|u_\phi(x)|\cdot |x|^{m}\le C,
$$
where the exponent $m$ is sharp. Using the relation $p=2+\frac{2m_1m}{m+1}$ and eliminating $m_1=m_2$ by virtue of \eqref{obstr} we obtain
$$
m=\frac{2n-p-2}{p-2}.
$$
By analogy with the harmonic polynomials in $\R{2}$, it is natural to think of $u_\phi(x)$ as a good candidate for the extremal case for two-zone solutions. Note also that $m$ depend on the dimension $n$ and the parameter $p$ only, but not on a particular choice of $\phi$. This makes plausible to believe that the above Conjecture is true.

\section{Polynomial solutions to \eqref{plap1}}\label{sec:pol}

The present paper arose in an attempt to construct homogeneous  polynomial solutions to the $p$-Laplace equation based on isoparametric polynomials. In 1980, John Lewis asked in \cite{Lewis80} to characterize all (non-linear) polynomial solutions $u$ to the $p$-Laplace equation for $p>1$ and $p\ne2$.  The nontrivial part of the problem is to characterize the possible homogeneous polynomial solutions. If the  degree of $u$ is even then one always has radially symmetric polynomial solutions, but in that case always $p<1$. If $p>1$ and $p\ne 2$ then it is known that there are no  homogeneous degree $m$ polynomial solutions to the $p$-Laplace equation in the following cases:
\begin{itemize}
\item[$\bullet$]
$n=2$ and any $m\ge 2$ \cite{Lewis80},
\item[$\bullet$]
$m=3$ and any $n\ge 2$ \cite{Tk16pLapl},
\item [$\bullet$]
$m=4$ and any $n\ge 2$, $m=5$ and $n=3$ \cite{Lewis2016}.
\end{itemize}

It is a common believe that the answer to Lewis' question is negative. Note that, heuristically, if a homogeneous solution would exist it would have some very distinguished (symmetric or extremal) properties. Isoparametric form \eqref{uharm} is  a natural candidate. But the result below shows that there are no isoparametric type examples.

 \begin{theorem}
 Let $p>1$, $p\ne 2$. Then there are no homogenous polynomial solution of \eqref{plaplace} satisfying \eqref{uharm} and $\deg u\ge2$.
 \end{theorem}

\begin{proof}
We argue by contradiction and assume that $u=f(s,t)\not\equiv 0$ is  homogenous degree $\deg u=N\ge2$ (in $x$) solution to \eqref{plaplace}. Then $m | N$, say $N=mk$ for some integer $k\ge1$ and $f(s,t)$ is itself a homogeneous polynomial in $(s,t)$ of homogeneous degree $k$. Let us denote it by
$$
f(s,t)=\sum_{j=0}^ka_js^{j}t^{k-j}\equiv s^kg(z),
$$
where $g=g(z)$ is a (nontrivial) polynomial in $z=t/s$ of degree $\le k$.
Then $g(z)$ satisfies  \eqref{geq}.  Factorizing  $g(z)=\prod_{i=1}^{r} (z-z_i)^{q_i}$ with $z_i$ pairwise distinct complex numbers, we obtain
$$
h:=\frac{g'}{g}=\sum_{i=1}^{r} \frac{q_i}{z-z_i},
$$
where each $q_i\in \mathbb{Z}^+$ is a positive integer ($=$ the multiplicity of $z_i$). It follows from \eqref{geq} that
\begin{equation}\label{geq1}
\begin{split}
b_1(z^2-1)^2(h'+h^2)h^2&+b_2(z^2-1)(h^2+1)+(b_3z+\nu)(z^2-1)h^3\\ &+b_4(z^2-1)h^2+(b_5z-\nu)h +b_6 =0,
\end{split}
\end{equation}

\textbf{Claim 1}: The set of zeros $\{z_1,\ldots,z_r\}$ does not contain $\pm1$. Indeed, arguing by contradiction, assume, for example, $z_1=1$. Suppose first that $q_1\ge 2$. Then the principal part of the corresponding Laurent decompositions at $z=1$ are
\begin{align*}
(z^2-1)^2(h^2+h')h^2&=\frac{4q_1^3(q_1-1)}{(z-1)^2}+O((z-1)^{-1}),\\
z(z^2-1)h^3&=\frac{2q_1^3}{(z-1)^2}+O((z-1)^{-1}),\\
(z^2-1)(h^2+h')&=\frac{2q_1(q_1-1)}{z-1}+O(1),\\ (z^2-1)h^2&=\frac{2q_1^2}{z-1}+O(1),\\
zh&=\frac{q_1}{z-1}+O(1).
\end{align*}
Combining the found relations with \eqref{geq1} and \eqref{alpha} implies that
$$
q_1^3(4b_1(q_1-1)+2(b_3+\nu))=0
$$
hence by the assumption $p>1$ it follows
$$
q_1=1-\frac{b_3+\nu}{2b_1}=1-\frac{p+m_1-1}{2(p-1)}
<1,
$$
a contradiction follows. Thus $q_1=1$ and $h=\frac{1}{z-1}+O(1)$ at $z=1$. Repeating the above argument  yields the Laurent expansion
\begin{align*}
z(z^2-1)h^3&=\frac{2}{(z-1)^2}+O(\frac{1}{z-1}),
\end{align*}
while the other terms in \eqref{geq1} either regular at $z=1$ or have the order $O(\frac{1}{z-1})$. This yields $b_3+\nu=0$, hence $p=1-m_1\le 0$, a contradiction again.

The same argument holds also true for $z=-1$. Thus, $z_i\ne \pm1.$

\textbf{Claim 2}: $q_i=1$ for all $i$. Indeed, if some $q_i>1$ then arguing as above
we obtain for the Laurent decompositions at $z=z_i$ that
\begin{align*}
(z^2-1)^2(h^2+h')h^2&=\frac{q_i^3(q_i-1)(z_i^2-1)^2}{(z-z_i)^4}+O((z-z_i)^{-3}),
\end{align*}
while other terms in \eqref{geq1} have singularity of order at most $O((z-z_i)^{-3})$. This yields $b_1=1-p=0$, a contradiction follows.

It follows from Claim~1 and 2 that the polynomial $g(z)$ has only simple roots, all distinct from $\pm1$. Equivalently,
\begin{equation}\label{equiv}
(z_i^2-1)g'(z_i)\ne0, \quad 1\le i\le r.
\end{equation}
Let us rewrite \eqref{geq} as a polynomial identity
\begin{equation}\label{PQ}
(z^2-1)g'^2P+gQ=0
\end{equation}
where
\begin{align*}
Q&=b_2(z^2-1)g''g+b_4(z^2-1)g'^2+b_5zg'g +b_6 g^2\\
P&=b_1g''(z^2-1)+(b_3z+\nu)g'.
\end{align*}
Setting $z=z_i$ in \eqref{PQ} and taking into account that $g(z_i)=0$   yields by virtue of \eqref{equiv} that $P(z_i)=0$. Thus, the polynomial $P$ vanishes whenever $g$ does, and also $\deg P\le \deg g$. Since the roots of $g$ are simple, we have  $P(z)=\lambda g(z)$ for some $\lambda\in \R{}$, i.e.
\begin{equation}\label{Pdiff}
b_1(z^2-1)g''+(b_3z+\nu)g'=\lambda g
\end{equation}
Substituting this identity into \eqref{PQ} yields $Q=-\lambda (z^2-1)g'^2$, thus after elimination of $(z^2-1)g''$ from the obtained relation by virtue of (\ref{Pdiff}) yields
$$
(b_6b_1+\lambda b_2) g^2+((b_5b_1-b_2 b_3)z-\nu b_2)) gg'+b_1(\lambda+b_4) (z^2-1)g'^2\equiv 0.
$$
Setting $z=z_i$ in the latter identity yields by virtue of \eqref{equiv}, $b_1\ne0$ and $g(z_i)=0$ that $\lambda=-b_4$, and therefore
$$
(b_6b_1-b_4 b_2) g+((b_5b_1-b_2 b_3)z-b_2\nu) g'\equiv 0.
$$
Since
$$
b_5b_1-b_2 b_3=k^2(2-p)\mu\ne0,
$$
we obtain $g(z)=C_1(z-a)^\beta$, where
$$
a=\frac{b_2\nu}{b_5b_1-b_2 b_3}, \quad \beta=\frac{b_4 b_2-b_6b_1}{b_5b_1-b_2 b_3},
$$
therefore either $g$ is linear (the case treated  in Section~\ref{sublinear}), or it has a single root of  multiplicity $\ge 2$, a contradiction again.
\end{proof}

\section{Biharmonic examples}\label{sec:some}

We finish this paper by a few more curious examples of biharmonic functions based on isoparametric polynomials. Let us consider the isoparametric ansats \eqref{uharm}. Then it follows from \eqref{laplce} that the harmonicity of $f$ is equivalent to the vanishing of the linear operator $C$ in \eqref{notation}, and biharmonic examples are obtained by the first iteration of \eqref{laplce}. A complete analysis of the obtained equation, though much simpler than in the nonlinear case \eqref{geq}, is beyond the scope of this article. We confine ourselves to a particular case
\begin{equation}\label{biharm}
u(x):=(|x|^m+\phi(x))^{\alpha}.
\end{equation}
Interesting that in contrast to the $p$-harmonic case,  the obtained below biharmonic examples involve isoparametric polynomials with  $m_1\ne m_2$.

\begin{proposition}
\label{pro:bihar}
Let $\phi\in \Iso_m(m_1,m_2)$ and $u$ is defined by \eqref{biharm}. Then $\Delta^2 u=0$ but $\Delta u\not\equiv  0$ iff
\begin{itemize}
\item[(i)] $m=2$ and $u$ is the fundamental solution of $\Delta^2$ in $\R{k}$ for some $k\le n$, or
\item[(ii)] $m=4$ and $(m_1,m_2)\in \{(1,4), (2,5), (4,7), (6,9)\}$.
    \end{itemize}
    In the latter case, the corresponding function $u(x)$ is biharmonic everywhere in $\R{n}$ outside  a minimal cone of codimension $m_2+1$.
\end{proposition}
\begin{remark}
The first case yields the well-known fundamental solutions, while the four examples obtained in (ii) are new, to the best of our knowledge.
Note  that for $m_1=1$ the corresponding solution is algebraic, and for $m_2=2,4,6$ \eqref{biharm} yields three \textit{rational} biharmonic functions. It also follows from the KFM table, see \cite[p.~178]{CecilRyan2015}, that for each pair $(m_1,m_2)\in \{(1,4), (2,5), (6,9)\}$ there is  essentially unique (up to an isometry of the ambient space) isoparametric quartic. In the exceptional case $(m_1,m_2)=(4,7)$ there exist two different isoparametric quartics.
\end{remark}
\begin{proof}
Let $v=|x|^m+\varphi$, such that $u=v^\alpha$. Then using \eqref{obstr}, \eqref{Muntzer1} and \eqref{Muntzer2}  we find
$$
|\nabla v|^2=m^2|x|^6+2m|x|^2\scal{x}{\nabla \phi(x)}+|\nabla \phi|^2=2m^2v|x|^2
$$
and
$$
\Delta v=\bigl(m(n+m-2)+\half{m^2}{2}(m_2-m_1)\bigr)|x|^2=m^2(m_2+1)|x|^2,
$$
therefore
$$
\Delta v^\alpha=\alpha v^{\alpha-1}\Delta v+\alpha(\alpha-1)|\nabla v|^2v^{\alpha-2}=\gamma_\alpha v^{\alpha-1}|x|^2,
$$
where $\gamma_\alpha=\alpha m^2(m_2+2\alpha-1)$. By the assumption $\Delta v^\alpha\not\equiv 0$, hence $\gamma_\alpha\ne0$. Then iterating the obtain relation we find
$$
\frac{1}{\gamma_\alpha}\Delta^2 v^\alpha=\Delta |x|^2v^{\alpha-1}=\biggl((m-2)(n-4+m(2\alpha-1))v +\gamma_{\alpha-1}|x|^4\biggr)v^{\alpha-2}.
$$
The latter expression identically vanishes if and only if
\begin{align}\label{bihh1}
(m-2)(n-4+m(2\alpha-1))&=0 \\
\gamma_{\alpha-1}=\alpha m^2(m_2+2\alpha-3)&=0.\label{bihh2}
\end{align}

First assume that $m=3$ or $6$. Then  $m_1=m_2$ (see Section~\ref{sec54}), therefore \eqref{bihh2} yields  $m_2=3-2\alpha$, hence $n=mm_2+2=m(3-2\alpha)+2$. On the other hand, \eqref{bihh1} gives $n=4+m(1-2\alpha)$. Eliminating $n$ yields $m=1$, a contradiction.

Next consider $m=2$. Then \eqref{bihh1} is fulfilled automatically, and \eqref{bihh2} gives $m_2=3-2\alpha$. It follows from \eqref{m2} that $\phi=|\xi|^2-|\eta|^2$, where $x=(\xi,\eta)\in \R{m_2+1}\times \R{m_1+1}$. But in that case, $u=(|x|^2+\phi)^{(3-m_2)/2}=C|\eta|^{3-m_2}$ is a function of $\eta\in \R{m_2+1}$ and it is the fundamental solution of $\Delta^2$ in $\R{m_2+1}$. This yields (i).

Finally, suppose $m=4$. Arguing similarly, we find $m_2=3+m_1$ and $\alpha=-m_1/2$. In particular, $(m_1,m_2)$ is distinct from the two exceptional pairs $(2,2)$, $(4,5)$ and $(7,8)$, thus the corresponding isoparametric quartic is of KFM type, i.e. it is congruent to \eqref{FKM}, see \cite{ChiBook}. Therefore, combining the obtained relation  with \eqref{m1m2m} yields $p=2q+2$, where the possible values $(p,q)$ are determined from the Hurwitz-Radon obstruction \eqref{qp}, i.e.
$$
q\le 1+\rho(2q+2).
$$
Since the Hurwitz-Radon function $\rho$ has the logarithmic growth, the latter inequality has only finitely many solutions. A simple analysis shows that the only possible solutions are $q\in\{1,2,3,5,7\}$. By \eqref{m1m2m}, $m_1=q-1$. If $m_1=0$ then $\alpha=0$, hence $u=const$. If $q\ge 2$ then $m_1\ge1$, thus we arrive at the four possible pairs of isoparametric parameters:
$$
(m_1,m_2)\in \{(1,4), \,(2,5), \,(4,7), \,(6,9)\},
$$
all realizable, see for example Table in section~4.3 in \cite{FKM} or \cite[p.~178]{CecilRyan2015}.

Next, since $\alpha<0$, the function $u(x)$ is well-defined everywhere in $\R{n}\setminus CM^-$, where $CM^-$ is the cone over
$$
M^-=\{x\in S^{n-1}:u(x)=0\}=\{x\in S^{n-1}:\phi(x)=-1\}.
$$
Then the desired claim follows from the fact that the focal varieties  $M^\pm:=\{x\in S^{n-1}:\phi(x)=\pm1\}$ are minimal submanifolds of the unit sphere $S^{n-1}$ of codimension $m_1+1$ for $M^+$ and $m_2+1$ for $M^-$, respectively, see \cite{Nomi}.

\end{proof}

\begin{example}[$(m_1,m_2)=(1,4)$]
Then $q=2$,   $p=2q+2=6$ and $n=2p=12$. The corresponding Clifford symmetric system in $\R{12}$ is given by
$$
A_1=\left(
        \begin{array}{cc}
         1_6 & 0 \\
          0 & -1_6\\
        \end{array}
      \right),
\quad
A_2=\left(
        \begin{array}{cc}
         0 & 1_6 \\
          1_6 & 0\\
        \end{array}
      \right),
$$
where $1_d$ is the $d\times d$-unit matrix. It follows from \eqref{FKM} and \eqref{biharm} that
$$
u=\frac{1}{2\sqrt{|\xi|^2|\eta|^2-\scal{\xi}{\eta}^2}}, \quad \xi=(x_1,\ldots,x_6),\quad \eta=(x_7,\ldots,x_{12})
$$
The obtained biharmonic function is well-defined is degree $-2$ homogeneous and well-defined everywhere in $\R{12}\setminus CM^-$, where the singular set $CM^-$ is the  $7$-dimensional minimal cone
$$
\Gamma=\{(x\cos \theta ,x\sin \theta )\in \R{12}: x\in \R{6} \quad \text{and}\quad \theta\in [0,2\pi]\}.
$$
\end{example}

\begin{example}[$(m_1,m_2)=(2,5)$] In this case, $(q,p)=(3,8)$ and $n=16$. Then any symmetric Clifford system with $(q,p)=(3,8)$ is geometric equivalent to the triple
$$
A_1=\left(
        \begin{array}{cccc}
         1_4 & 0 & 0 &0\\
         0 & -1_4 & 0 &0\\
         0 & 0 & 1_4 &0\\
         0 & 0 & 0 &-1_4
        \end{array}
      \right),
\quad
A_2=\left(
        \begin{array}{cccc}
         0 & 1_4 & 0 &0\\
         1_4 & 0 & 0 &0\\
         0 & 0 & 0 & 1_4\\
         0 & 0 & 1_4 & 0
        \end{array}
      \right),
\quad A_3=\left(
        \begin{array}{cccc}
         0 & 0 & 0 &1_4\\
         0 & 0 & 1_4 &0\\
         0 & -1_4 & ´0 &0\\
         1_4 & 0 & 0 &0
        \end{array}
      \right),
$$
Then the corresponding biharmonic function is rational:
$$
u=\frac{1}{(|\xi_1|^2+|\xi_3|^2)(|\xi_2|^2+|\xi_4|^2)- (\scal{\xi_1}{\xi_2}+ \scal{\xi_3}{\xi_4})^2- (\scal{\xi_1}{\xi_4}- \scal{\xi_2}{\xi_3})^2},
$$
where $\xi=(x_{i+1},\ldots,x_{i+4})$, $i=1,2,3,4.$
The function $u$ is biharmonic outside  the minimal cone $CM^-\subset\R{16}$ of codimension $m_2+1=6$.

\end{example}


\bibliographystyle{plain}

\def\cprime{$'$}


\end{document}